	\documentclass[12pt]{amsart}

 	\usepackage{times} 
	\usepackage{amssymb}  
	\usepackage{amscd}
	\usepackage{amsfonts}
	\usepackage{amsmath}
	\usepackage{amsthm}

	\theoremstyle{plain}
	\newtheorem{prop}{Proposition}[section]
	\newtheorem{thm}[prop]{Theorem}
	\newtheorem{cor}[prop]{Corollary}    
	\newtheorem{lem}[prop]{Lemma}

	\theoremstyle{definition}

	\numberwithin{equation}{section}

\newcommand{\dqbin}[2]{\displaystyle\genfrac{[}{]}{0pt}{}{#1}{#2}}

\newcommand{\N}{\mathbb{N}}
\newcommand{\Z}{\mathbb{Z}}
\newcommand{\F}{\mathbb{F}}

\newcommand{\Q}{\mathbb{Q}}
\newcommand{\LL}{\Lambda}

\newcommand{\mac}[1]{\widetilde{H}_{#1}}
\newcommand{\kostka}{\widetilde{K}}

\begin{document}

\title{Some Plethystic Identites and Kostka-Foulkes polynomials.}

\author{Mahir Bilen Can}
%\address{Tulane University}
%\email{mcan@uwo.ca}

\date{}

\maketitle

\section{\textbf{Introduction.}}

	Symmetric functions  $\{E_{n,k}(X)\}_{k=1}^n$, defined  by the Newton interpolation
	\begin{equation*}
	e_n[X\frac{1-z}{1-q}]=\sum_{k=1}^n (z;q)_k \frac{E_{n,k}(X)}{(q;q)_k}
	\end{equation*}
	plays an important role in the Garsia-Haglund proof of the $q,t$-Catalan conjecture, \cite{GaHag02}.

	Let $\Lambda^n_{\Q(q,t)}$ be the space of symmetric functions of degree $n$, 
	over the field of rational functions $\Q(q,t)$, and let 
	$\nabla: \Lambda^n_{\Q(q,t)} \rightarrow \Lambda^n_{\Q(q,t)}$  be the Garsia-Bergeron operator.

	By studying recursions, Garsia and Haglund show that the coefficient of the elementary 
	symmetric function $e_n(X)$ in the image $\nabla( E_{n,k}(X))$ of $E_{n,k}(X)$ 
	is equal to the following combinatorial summation
	\begin{equation}\label{E:keyequation}
	\langle \nabla (E_{n,k}(X)), e_n(X)\rangle = \sum_{\pi \in D_{n,k}} q^{area(\pi)}t^{bounce(\pi)},
	\end{equation}
	where $D_{n,k}$ is the set of all Dyck paths with initial $k$ North steps followed by an East step. 
	Here $area(\pi)$ and $bounce(\pi)$ are two numbers associated with a Dyck path $\pi$.  
	It is conjectured in \cite{Hag07}, more generally, that the $\nabla E_{n,k}(X) $ are ``Schur positive.''

 	In \cite{CanLoehr06}, using (\ref{E:keyequation}), Can and Loehr  prove the $q,t$-Square 
	conjecture of the Loehr and Warrington \cite{LoehrWarrington}.

	The aim of this article is to understand the functions $\{E_{n,k}(X)\}_{k=1}^n$ better.  
	We prove that the vector subspace generated by the set $\{E_{n,k}(X)\}_{k=1}^n$  of the space 
	$\Lambda^n_{\Q(q)}$  of degree $n$ symmetric functions over the field $\Q(q)$, 
	is equal to the subspace generated by 
	$$
	\{ s_{(k,1^{n-k})}[X/(1-q)]\}_{k=1}^n,
	$$ 
	Schur functions of hook shape, plethystically evaluated at $X/(1-q)$.

	In particular, we determine explicitly the transition matrix and its inverse from 
	$\{E_{n,k}(X)\}_{k=1}^n$ to $\{ s_{(k,1^{n-k})}[X/(1-q)]\}_{k=1}^n $. 
	The entries of the matrix turns out to be cocharge Kostka-Foulkes polynomials.

	We find the expansion of $E_{n,k}(X)$ into the Hall-Littlewood basis, and as a corollary 
	we recover a closed formula for the  cocharge Kostka-Foulkes polynomials 
	$\widetilde{K}_{\lambda,\mu}(q)$ when $\lambda$ is a hook shape;
	\begin{equation*}
	 \widetilde{K}_{(n-k,1^{k}) \mu}(q)=(-1)^{k} \sum_{i=0}^{k} (-1)^i q^{i\choose 2} \dqbin{r}{i}.
	\end{equation*}
	Here, $\mu$ is a partition of $n$ whose first column is of height $r$.

	\section{\textbf{Background.}}                        
 
	\subsubsection{Notation} A partition $\mu$ of $n\in \Z_{>0}$, denoted $\mu \vdash n$, 
	is a nonincreasing sequence $\mu_1\geq \mu_2\geq \ldots \geq \mu_k >0$ of numbers such that 
	$\sum \mu_i = n$. The conjugate partition $\mu'=\mu_1'\geq \ldots \geq \mu_s'>0$ is defined by setting 
	$\mu_i'= |\{\mu_r:\ \mu_r \geq i \}|.$

	 $Par(n,r)$ denotes the set of all partitions $\mu \vdash n$ whose biggest part is equal to $\mu_1=r$. 

	We identify a partition $\mu$ with its Ferrers diagram, in French notation. Thus, 
	if the  parts of $\mu$ are $\mu_1\geq \mu_2\geq \cdots\geq \mu_k>0$,  then 
	the corresponding Ferrers diagram have  $\mu_i$ lattice cells in the $i^{th}$ row
	(counting from bottom to up).

	Following Macdonald, \cite{Mac95} the arm, leg, coarm and coleg
	of a lattice square $s$ are the parameters $a_\mu(s),l_\mu(s),a_\mu'(s)$ and $l_\mu'(s)$ 
	giving the number of cells of $\mu$ that are respectively strictly  EAST, NORTH, WEST and 
	SOUTH of $s$ in $\mu$.

	Given a partition $\mu=(\mu_1,\mu_2,\ldots ,\mu_k)$, we set
	\begin{equation}\label{E:legsum}
	n(\mu)=\sum_{i=1}^k (i-1)\mu_i= \sum_{s\in \mu}\, l_\mu(s) .
	\end{equation}
	We also set 
	\begin{equation}\label{E:denom}
	\widetilde{h}_{\mu}(q,t)=\prod_{s\in \mu} (q^{a_{\mu}(s)}-t^{l_{\mu}(s)+1}\ )\hspace{.2in}  \text{and}\hspace{.2in}
	\widetilde{h'}_{\mu}(q,t)=\prod_{s\in \mu} (t^{l_{\mu}(s)}-q^{a_{\mu}(s)+1}).
	\end{equation}

	Let $\F$ be a field, and let $X=\{x_1,x_2,...\}$ be an alphabet (a set of  indeterminates). 
	The algebra of symmetric functions over $\F$ with the variable set $X$  is denoted by $\LL_{\F}(X)$. 

	If $\Q\subseteq \F$, it is well known that $\LL_{\F}(X)$ is freely generated by the set of power-sum 
	symmetric functions 
	 \begin{equation*}
	 \{ p_r(X):\ r=1,2,...\  \text{and},\ p_r(X)=x_1^r+x_2^r+\cdots \}.
	 \end{equation*}

	The algebra, $\LL_{\F}(X)$ has a natural grading (by degree).
	\begin{equation*}
	\LL_{\F}(X)=\bigoplus_{n\geq 0} \LL_{\F}^n(X),
	\end{equation*}
	where $\LL_{\F}^n(X)$ is the space of homogenous symmetric functions of degree $n$.

	A basis for the vector space $\LL_{\F}^n(X)$ is give by the set 
	$\{p_{\mu}\}_{\mu \vdash n}$, 
	\begin{equation}
	p_{\mu}(X) = \prod_{i=1}^k p_{\mu_i}(X),\  \text{where}\ \mu=\sum_{i=1}^k \mu_i.
	\end{equation}

	Another basis for  $\LL_{\F}^n(X)$  is given by the Schur functions $\{s_{\mu}(X)\}_{\mu \vdash n}$, 
	where $s_{\mu}(X)$ is defined as follows. Let
	\begin{equation}
	e_n(X)=\sum_{1\leq i_1 < \cdots < i_n} x_{i_1}x_{i_2}\cdots x_{i_n}
	\end{equation}
	be the $n$'th elementary symmetric function. If $\mu=\sum_{i=1}^k \mu_i$, then
	\begin{equation}
	s_{\mu}(X) =\det (e_{\mu_i'-i+j}(X))_{1\leq i,j \leq m},
	\end{equation}
	where  $\mu_i'$ is the $i$'th part of the conjugate partition $\mu'=(\mu_1',...,\mu_{l}')$ and $m\geq l$.

	\subsubsection{Plethysm}

	For the purposes of this section, we represent an alphabet $X=\{x_1,x_2,...\}$ 
	as a formal sum $X=\sum x_i$. Thus, if  $Y=\sum y_i$ is another alphabet, then
	\begin{equation}
	XY= (\sum x_i)(\sum y_i)  =\sum_{i,j} x_iy_j= \{x_iy_j\}_{i,j\geq 1},
	\end{equation} 
	and 
	\begin{equation}
	X+Y=(\sum_i x_i)+ (\sum_j y_j)= \{x_i,y_j\}_{i,j \geq 1}.
	\end{equation} 

	The formal additive inverse, denoted $-X$, of an alphabet $X=\sum x_i$ is defined so that $-X+X=0$.

	In this vein, if $p_k(X)=\sum_{k\geq 1} x_i^k$ is a power sum symmetric function, we define
	\begin{eqnarray}\label{E:powerplethysm}
	p_k[XY] &=& p_k[X]p_k[Y]\\
	p_k[X+Y] &=& p_k[X] + p_k[Y]\\
	p_k[-X] &=& -p_k[X].
	\end{eqnarray}

	This operation is called \textit{plethysm.}  Since $\LL_{\F}$ is freely generated by the power sums, 
	the plethysm operator can be extended to the other symmetric functions.  In fact, using plethysm, 
	one defines the following bases for $\LL_{\Q(q)}^n$ and $\LL_{\Q(q,t)}^n$, respectively.

	\textbf{Theorem-Definition 1.} (\textit{cocharge Hall-Littlewood polynomials})\\
	There exists a basis $\{\mac{\mu}(X;q)\}_{\mu \vdash n}$ for the vector space 
	$\LL_{\Q(q)}^n$, which  is uniquely characterized by the properties
	\begin{enumerate}
	\item $\mac{\mu}(X;q)\in \Z[q]\{s_{\lambda}: \lambda \geq \mu\}$,
	\item $\mac{\mu}[(1-q)X;q]\in \Z[q]\{s_{\lambda}: \lambda \geq \mu'\}$,
	\item $\langle \mac{\mu}(X;q),s_{(n)} \rangle=1.$
	\end{enumerate}

	\textbf{Theorem-Definition 2.} (\textit{Modified Macdonald polynomials})\\
	There exists a basis $\{\mac{\mu}(X;q,t)\}_{\mu \vdash n}$ for the vector space $\LL_{\Q(q,t)}^n$, 
	which is  uniquely characterized by the properties
	\begin{enumerate}
	\item $\mac{\mu}(X;q,t)\in \Z[q]\{s_{\lambda}: \lambda \geq \mu\}$,
	\item $\mac{\mu}[(1-q)X;q,t]\in \Z[q]\{s_{\lambda}: \lambda \geq \mu'\}$,
	\item $\langle \mac{\mu}[X(1-t);q,t], s_{(n)} \rangle=1.$
	\end{enumerate}

	It follows from these Theorem-Definitions that 
	\begin{eqnarray}
	\label{E:specialization1}
	\mac{\mu}(X;0,t) &=&\mac{\mu}(X;t),\\
	\label{E:specialization2}
	\mac{\mu}(X;q,t) &=& \mac{\mu'}(X;t,q).
	\end{eqnarray}

	\subsubsection{Kostka-Foulkes and Kostka-Macdonald polynomials}

	Let $\mac{\mu}(X;q)=\sum_{\lambda} \kostka_{\lambda \mu}(q) s_{\lambda}$, and
	$\mac{\mu}(X;q,t)=\sum_{\lambda} \kostka_{\lambda \mu}(q,t) s_{\lambda}$ be, respectively, 
	the Schur basis expansions of the Hall-Littlewood and Macdonald symmetric functions. 
	The coefficients of the Schur functions are called, respectively, the \textit{cocharge Kostka-Foulkes polynomials}, 
	and the \textit{modified Kostka-Macdonald polynomials}. 
	It is known that $ \kostka_{\lambda \mu}(q,t), \kostka_{\lambda \mu}(q) \in \N[q,t]$.

	It follows from equation (\ref{E:specialization1}) and the Schur basis expansions that 
	\begin{equation}\label{E:kostkaspecialization}
	\kostka_{\lambda \mu}(0,t) = \kostka_{\lambda \mu}(t).
	\end{equation}

	\subsubsection{Cauchy Identities}

	Let $X=\sum x_i$ be an alphabet, and let 
	$\Omega[X] = \exp(\sum_{k=1}^{\infty} p_k(X)/k).$
	Then, 
	\begin{eqnarray}\label{E:OMEGA}
	\Omega[X] &=& \prod_i \frac{1}{1-x_i} = \sum_{n=0}^{\infty} s_n(X),\\
	\Omega[X] &=& \prod_i (1-x_i)= \sum_{n=0}^{\infty} s_{1^n}(X).
	\end{eqnarray}

	If $Y=\sum y_i$ is another alphabet, then 
	\begin{equation} \label{E:Cauchy1}
	e_n[XY]=\sum_{\mu \vdash n}s_{\mu}[X]s_{\mu'}[Y],
	\end{equation}

	\begin{equation}\label{E:Cauchy2}
	e_n[XY]= \sum_{\mu \vdash n} \frac{\mac{\mu}[X;q,t] \mac{\mu}[Y;q,t]}{\widetilde{h}_{\mu}(q,t) \widetilde{h'}_{\mu}(q,t)},
	\end{equation}
	where $\widetilde{h}_{\mu}(q,t)$ and $\widetilde{h'}_{\mu}(q,t)$ are as in (\ref{E:denom}).

	\begin{equation}\label{E:Cauchy3}
	s_{\mu}[1-z]=\left\{\begin{array}{cc}
	(-z)^k(1-z) & \mbox{if }\mu=(n-k,1^k),\\
	0 & \mbox{otherwise.} 
	\end{array}\right.
	\end{equation}

	\subsubsection{Cauchy's $q$-binomial theorem}

	Let $(z;q)_k  = (1-z)(1-qz)\cdots(1-q^{k-1}z)$, and let 
	\begin{equation*}
	\dqbin{k}{r} = \frac{(q;q)_k}{(q;q)_r(q;q)_{k-r}}.
	\end{equation*} 

	Then, the Cauchy $q$-binomial theorem states that 
	\begin{equation}\label{E:cauchyqbinom}
	(z;q)_k = \sum_{r=0}^k z^r  (-1)^r e_r[1,q,...,q^{k-1}] = 
	\sum_{r=0}^k z^r q^{ {r\choose 2} } (-1)^r \dqbin{k}{r}.
	\end{equation}

	\section{\textbf{Symmetric functions} $E_{n,k}(X)$.}

	The family $\{E_{n,k}(X)\}_{k=1}^n$ of symmetric functions are defined by the plethystic identity 
	\begin{equation}\label{E:basicexp}
	e_n[X\frac{1-z}{1-q}]=\sum_{k=1}^n \frac{(z;q)_k}{(q;q)_k}E_{n,k}(X).
	\end{equation}

	Let $0\leq k\leq r $, and let 
	\begin{equation}
	T_{k+1,r}=(-1)^{k} \sum_{i=0}^k (-1)^i q^{i\choose 2}
	\dqbin{r}{i}.
	\end{equation}

	\begin{prop}\label{P:SchurtoE}
	For $k=0,...,n-1$,
	\begin{equation}
	s_{(k+1,1^{n-k-1})}[X/(1-q)]  =
	 \sum_{r=k+1}^n  T_{k+1,r} \frac{E_{n,r}(X)}{(q;q)_r}.
	\end{equation}
	\end{prop}

	\begin{proof}
	Using the Cauchy $q$-binomial theorem, we see that the coefficient of $(-z)^k$ on the right hand side of 
	(\ref{E:basicexp}) is  

	\begin{equation}
	q^{{k\choose 2}} \sum_{i=0}^{n-k} \dqbin{k+i}{k}
	\frac{E_{n,k+i}}{(q;q)_{k+i}}.
	\end{equation}

	On the other hand, by the identities (\ref{E:Cauchy1}) and (\ref{E:Cauchy3}), 
	\begin{eqnarray*}
	e_n[X\frac{1-z}{1-q}] &=& \sum_{\lambda} s_{\lambda}[\frac{X}{1-q}]
	s_{\lambda'}[1-z]\\
	&=&\sum_{\lambda' = (n-r,1^r)}
	s_{\lambda}[\frac{X}{1-q}] (-z)^r(1-z) \\
                     &=& \sum_{r=0}^{n-1}
                     s_{(r+1,1^{n-r-1})}[\frac{X}{1-q}]
                     (-z)^r(1-z)\\
                     &=& s_{1^n}[\frac{X}{1-q}]+
	(-z)(s_{1^n}[\frac{X}{1-q}]+s_{2,1^{n-2}}[\frac{X}{1-q}])+ \cdots
	+(-z)^n s_n[\frac{X}{1-q}].
	\end{eqnarray*}
	Comparing the coefficient of $(-z)^k$ gives, for $k\geq 1$,
	\begin{equation}\label{E:consec}
	q^{{k\choose 2}} \sum_{i=0}^{n-k} \dqbin{k+i}{k}_q
	\frac{E_{n,k+i}}{(q;q)_{k+i}} =
	s_{k,1^{n-k}}[\frac{X}{1-q}]+s_{k+1,1^{n-k-1}}[\frac{X}{1-q}],
	\end{equation}
	and
	\begin{equation}\label{E:single}
	\sum_{i=1}^{n} \frac{E_{n,i}}{(q;q)_{i}} = s_{1^n}[\frac{X}{1-q}].
	\end{equation}

	We take the alternating sums of the equations (\ref{E:consec}) and (\ref{E:single}) to get
	\begin{equation*}
	s_{k+1,1^{n-k-1}}[X/(1-q)]= (-1)^k \left( \sum_{j=1}^{n}
	\frac{E_{n,j}}{(q;q)_{j}} \right) + \sum_{i=1}^k (-1)^{k+i} \left(
	q^{ i\choose 2}  \sum_{j=0}^{n-i} \dqbin{i+j}{i} \frac
	{E_{n,i+j}}{(q;q)_{i+j}} \right).
	\end{equation*}

	By collecting $E_{n,k}(X)$'s, and using (\ref{E:cauchyqbinom}) we obtain

	\begin{equation*}
	s_{(k+1,1^{n-k-1})}[X/(1-q)]  =
	 \sum_{r=k+1}^n  T_{k+1,r} \frac{E_{n,r}(X)}{(q;q)_r}.
	\end{equation*}
	
	\end{proof}

	Let $S$ and $E$ be the matrices
	\begin{equation*}
	S= \begin{pmatrix}
	s_{1^{n}}[X/(1-q)] \\
	s_{2,1^{n-1}}[X/(1-q)] \\
	\vdots \\
	s_{n}[X/(1-q)] 
	\end{pmatrix}\ \text{and}\ 
	E= \begin{pmatrix}
	\frac{E_{n,1}}{(q;q)_{1}}\\
	\frac{E_{n,2}}{(q;q)_{2}}\\
	\vdots \\
	\frac{E_{n,n}}{(q;q)_{n}} 
	\end{pmatrix},
	\end{equation*}
	respectively, and let $T$ be the transition matrix from $E$ to $S$, so that $S=TE$. 
	Then, $T$ is an upper triangular matrix with the $k+1,r$'th entry
	\begin{equation*}
	T_{k+1,r}=(-1)^{k} \sum_{i=0}^k (-1)^i q^{i\choose 2}
	\dqbin{r}{i}.
	\end{equation*}
	For example, when $n=5$, 
	\begin{equation*}
	T=\left( \begin {array}{ccccc} 1&1&1&1&1\\\noalign{\medskip}0&q&
	 \left( q+1 \right) q& \left( {q}^{2}+q+1 \right) q& \left( {q}^{3}+{q
	}^{2}+q+1 \right) q\\\noalign{\medskip}0&0&{q}^{3}&{q}^{3} \left( {q}^
	{2}+q+1 \right) &{q}^{3} \left( {q}^{4}+{q}^{3}+2\,{q}^{2}+q+1
	 \right) \\\noalign{\medskip}0&0&0&{q}^{6}&{q}^{6} \left( {q}^{3}+{q}^
	{2}+q+1 \right) \\\noalign{\medskip}0&0&0&0&{q}^{10}\end {array}
	 \right).
	\end{equation*}
	Then,
	\begin{equation*}
	T^{-1}= 
	\left( \begin {array}{ccccc} 1&-{q}^{-1}&{q}^{-2}&-{q}^{-3}&{q}^{-4}
	\\\noalign{\medskip}0&{q}^{-1}&-{\frac {q+1}{{q}^{3}}}&{\frac {{q}^{2}
	+q+1}{{q}^{5}}}&-{\frac {{q}^{3}+{q}^{2}+q+1}{{q}^{7}}}
	\\\noalign{\medskip}0&0&{q}^{-3}&-{\frac {{q}^{2}+q+1}{{q}^{6}}}&{
	\frac {{q}^{4}+{q}^{3}+2\,{q}^{2}+q+1}{{q}^{9}}}\\\noalign{\medskip}0&0
	&0&{q}^{-6}&-{\frac {{q}^{3}+{q}^{2}+q+1}{{q}^{10}}}
	\\\noalign{\medskip}0&0&0&0&{q}^{-10}\end {array} \right).
	\end{equation*}

	\begin{prop}\label{P:inversematrix}
	$T^{-1}$ is (necessarily) upper triangular, and its $k+1,r$'th entry is equal to 
	\begin{equation}
	(T^{-1})_{k+1,r}= (-1)^{r-k}q^{-r(k+1)} T_{k+1,r}.
	\end{equation}
	\end{prop}
	
	\begin{proof}
	Let $L$ be the upper triangular matrix with the $k+1,r$'th entry
	\begin{equation*}
	L_{k+1,r}=(-1)^{r-k}q^{-r(k+1)} T_{k+1,r}\hspace{.1in} \text{for}\ r>k.
	\end{equation*}
	Clearly, $TL$ is an upper triangular matrix, and the $i+1,j$'th entry of $TL$ is 
	\begin{equation}
	(TL)_{i+1,j} = \sum_{k=1}^n T_{i+1,k} L_{k,j}.
	\end{equation}
	It is straightforward to check that $(TL)_{i+1,i+1}=1$.
	We use induction on $j$ to prove that for all $i+1<j$, $(TL)_{i+1,j}=0$. 
	So, we assume that for all $i+1<j$, $(TL)_{i+1,j}=0$, and we are going prove that for all 
	$i+1<j+1$, $(TL)_{i+1,j+1}=0$.

	First of all, using the $q$-binomial identity
	\begin{equation}\label{E:qidentity}
	\dqbin{r}{m}=\dqbin{r-1}{m}+\dqbin{r-1}{m-1}q^{r-m},\ \text{for}\ m\geq 0,
	\end{equation}
	it is easy to show that 
	
	\begin{equation}\label{E:Trecursion}
	T_{i+1,k}=T_{i+1,k-1}+q^i T_{i,k-1}.
	\end{equation} 

	It follows that 

	\begin{equation}\label{E:Lrecursion}
	L_{k+1,j+1}=-q^{-(k+1)} L_{k+1,j}+q^{-k} L_{k,j-1}.
	\end{equation} 

	Therefore, 

	\begin{eqnarray*}
	\sum_{k=i+1}^{j+1} T_{i+1,k} L_{k,j+1} 
	&=& \sum_{k=i+1}^{j+1} T_{i+1,k} (-q^{-(k+1)} L_{k,j}+q^{-k} L_{k-1,j-1})\\
	&=& \sum_{k=i+1}^{j+1} -q^{-k} T_{i+1,k} L_{k,j} + \sum_{k=i+1}^{j+1} 
	q^{-(k-1)} T_{i+1,k} L_{k-1,j}. 
	\end{eqnarray*}

	Using (\ref{E:Trecursion}) in the last summation, we have
	
	\begin{eqnarray*}
	\sum_{k=i+1}^{j+1} T_{i+1,k} L_{k,j+1} 
	&=& \sum_{k=i+1}^{j+1} -q^{-k} T_{i+1,k} L_{k,j}+ \sum_{k=i+1}^{j+1} 
	q^{-(k-1)} T_{i+1,k-1} L_{k-1,j}\\ &+& \sum_{k=i+1}^{j+1} 
	q T_{i+1,k-1} L_{k-1,j}.
	\end{eqnarray*}

	After rearranging the indices, and using the induction hypotheses, the right hand side of the equation simplifies to 0.
	Therefore, the proof is complete.

	\end{proof}

	\begin{cor}
	Let $A\subseteq \Lambda^n_{\Q(q)}(X)$ be the $n$-dimensional subspace generated by the set 
	$\{E_{n,k}(X)\}_{k=1}^n$, and let $B\subseteq  \Lambda^n_{\Q(q)}(X)$ be the 
	$n$-dimensional subspace generated by $\{ s_{k,1^{n-k}}[\frac{X}{1-q}] \}_{k=1}^n$. Then, $A=B$.
	\end{cor}

	\begin{proof}
	It is clear by Proposition \ref{P:inversematrix} that $A=B$. The dimension claim follows from 
	Proposition \ref{P:EHall} below. 
	\end{proof}

	The expression $(-1)^n p_n = \sum_{k=0}^{n-1} (-1)^k s_{k+1,1^{n-k-1}}$ is the bridge between 
	Schur functions of hook type with the power sum symmetric functions.  By the linearity of plethysm we have   
	\[
	(-1)^n p_n[X]/(1-q^n)= (-1)^n p_n[X/(1-q)]= \sum_{k=0}^{n-1} (-1)^k s_{k+1,1^{n-k-1}}[X/(1-q)],
	\]
	and therefore
	\begin{equation}\label{E:pow}
	(-1)^n p_n= (1-q^n) \sum_{k=0}^{n-1}
	(-1)^k s_{k+1,1^{n-k-1}}[X/(1-q)].
	\end{equation}

	\begin{cor}\label{T:powersum} For all $n\geq 1$, 
	\begin{equation}
	(-1)^n p_n =  \sum_{r=1}^n
	\frac{1-q^n}{1-q^r} E_{n,r}.
	\end{equation}
	\end{cor}

	\begin{proof}
	By Proposition \ref{P:SchurtoE} and (\ref{E:pow}) we get
	\begin{equation}
	(-1)^n p_n  =(1-q^n) \sum_{k=0}^{n-1} \sum_{r=k+1}^n (-1)^kT_{k+1,r}\frac{E_{n,r}(X)}{(q;q)_r}.
	\end{equation}

	By rearranging the summations and using the Cauchy's $q$-binomial theorem once more, we finish the proof. 
	%We finish the proof by rearranging the summations, and using the following 
	%application of the Cauchy's $q$-binomial theorem;
	%\begin{equation*}
	%\left( 
	%\frac{ \sum_{i=0}^{r-1} (r-i) (-1)^i q^{i\choose 2}
	%\dqbin{r}{i}}{(q;q)_r} \right) = \frac{1}{1-q^r},\ \text{for}\ r\geq 1.
	%\end{equation*}
	\end{proof}

	\section{\textbf{Hall-Littlewood expansion.}}

	\begin{prop}\label{P:EHall}
	For $k=1,...,n$,
	\begin{equation*}
	\frac{E_{n,k}(X)}{(q;q)_k}=  \sum_{\mu \in Par(n,k)} 
	\frac{\mac{\mu'}(X;q)}{\widetilde{h}_{\mu}(q,0) \widetilde{h'}_{\mu}(q,0)} = \sum_{\mu \in Par(n,k)} 
	\frac{\mac{\mu'}(X;q)}{(-q)^n   q^{ 2 n(\mu')}  \prod_{s\in \mu,\l_{\mu}(s) = 0}(1-q^{-a_{\mu}(s)-1}) }.
	\end{equation*}
	\end{prop}

	\begin{proof}
	Let $Y=(1-t)(1-z)$. Then, by the Cauchy identity (\ref{E:Cauchy2}), we have 
	\begin{equation}\label{E:CauchyEE}
	\sum_{k=1}^n (z;q)_k \frac{E_{n,k}(X)}{(q;q)_k}= \sum_{\mu \vdash n} 
	\frac{\mac{\mu}[X;q,t] \mac{\mu}[(1-t)(1-z);q,t]}	{\widetilde{h}_{\mu}(q,t) \widetilde{h'}_{\mu}(q,t)}.	
	\end{equation}

	The left hand side of the equation (\ref{E:CauchyEE}) is independent of the variable $t$. 
	Since $\widetilde{h}_{\mu}(q,0) \neq 0$, and since  $\widetilde{h'}_{\mu}(q,0)\neq 0$, 
	we are allowed to make the substitution $t=0$ on both sides of the equation.

	Note that  
	\begin{eqnarray}
	\widetilde{h}_{\mu}(q,0) \widetilde{h'}_{\mu}(q,0) &=& 
	\prod_{s\in \mu} q^{a_{\mu}(s)} \prod_{s\in \mu,\l_{\mu}(s) \neq 0}(-q^{a_{\mu}(s)+1}) 
	\prod_{s\in \mu,\l_{\mu}(s) = 0}(1-q^{a_{\mu}(s)+1}) \\
	&=& (-q)^n  \prod_{s\in \mu} q^{2a_{\mu}(s)}  \prod_{s\in \mu,\l_{\mu}(s) = 0}\frac{1-q^{a_{\mu}(s)+1}}{-q^{a_{\mu}(s)+1}} \\
	&=& (-q)^n  \prod_{s\in \mu} q^{2a_{\mu}(s)}  \prod_{s\in \mu,\l_{\mu}(s) = 0}(1-q^{-a_{\mu}(s)-1}) \\
	\label{E:prod} &=& (-q)^n   q^{ 2 n(\mu')}  \prod_{s\in \mu,\l_{\mu}(s) = 0}(1-q^{-a_{\mu}(s)-1}) .
	\end{eqnarray}
	The equality (\ref{E:prod}) follows from (\ref{E:legsum}).

	Using the Schur expansion 
	$\mac{\mu}(X;q,t)=\sum_{\lambda} \kostka_{\lambda \mu}(q,t) s_{\lambda}$, we see that the plethystic substitution 
	$X \rightarrow (1-z)$, followed by the evaluation at $t=0$ is the same as the evaluation $\mac{\mu}(X,q,0)$ at $t=0$,
	followed by the plethystic substitution $X\rightarrow (1-z)$.  Also, by Corollary 3.5.20 of \cite{Hai03}, 
	we know that 
	\begin{equation*}
	\mac{\mu}[1-z;q,t]=\Omega[-zB_{\mu}],\ \text{where}\  B_{\mu}=\sum_{i \geq 1} t^{i-1}\frac{1-q^{\mu_i}}{1-q}.
	\end{equation*}

	Therefore,   
	\begin{eqnarray}
	\mac{\mu}[1-z;q,0] &=& \Omega[-z B_{\mu}]|_{t=0}\\
	&=& \Omega[-z(1+q+\cdots + q^{\mu_1-1})]\\
	&=& \prod_{i=0}^{\mu_1-1} (1-zq^i)\\
	\label{E:plethspec}&=& (z;q)_{\mu_1}.
	\end{eqnarray}

	It follows from (\ref{E:specialization1}) and (\ref{E:specialization2}) that 
	\begin{equation}\label{E:spec}
	\mac{\mu}(X;q,0)=\mac{\mu'}(X;q).
	\end{equation}

	By combining (\ref{E:CauchyEE}), (\ref{E:prod}), (\ref{E:plethspec}) and (\ref{E:spec}), we get 
	\begin{equation}\label{E:HallE}
	\sum_{k=1}^n (z;q)_k \frac{E_{n,k}(X)}{(q;q)_k} = \frac{1}{(-q)^n} \sum_{\mu \vdash n} (z;q)_{\mu_1} 
	\frac{ \mac{\mu'}(X;q)} {q^{2 n(\mu')} \prod_{s\in \mu,\l_{\mu}(s) = 0}(1-q^{-a_{\mu}(s)-1})}.
	\end{equation}

	By comparing the coefficient of $(z;q)_k$ in (\ref{E:HallE}), we find that
	\begin{equation}
	\frac{E_{n,k}(X)}{(q;q)_k}= \sum_{\mu \in Par(n,k)}   \frac{\mac{\mu'}(X;q)}{(-q)^n
	q^{ 2 n(\mu')}  \prod_{s\in \mu,\l_{\mu}(s) = 0}(1-q^{-a_{\mu}(s)-1})}.
	\end{equation}
	
	Hence, the proof is complete.

	\end{proof}

	\begin{lem}\label{L:GaHaTes}
	Let $\lambda \vdash n$ be a partition of $n$. Then,
	\begin{equation}\label{E:schurmac}
	s_{\lambda}[\frac{X}{(1-q)(1-t)}] =  \sum_{\mu\vdash n}\frac{\widetilde{K}_{\lambda'
	\mu}(q,t)\mac{\mu}(X;q,t)}{\widetilde{h}_{\mu}(q,t)\widetilde{h'}_{\mu}(q,t)}.
	\end{equation}
	\end{lem}
	\begin{proof}
	This follows from Theorem 1.3 of \cite{GaHaTes}.
	\end{proof}

	\begin{cor}\label{C:schurHall}
	Let $\lambda \vdash n$ be a partition. Then, 
	\begin{equation}\label{E:schurmac}
	s_{\lambda}[\frac{X}{1-q}] =  \sum_{\mu}  \frac{\widetilde{K}_{\lambda'
	\mu'}(q)\mac{\mu'}(X;q)}{\widetilde{h}_{\mu}(q,0) \widetilde{h'}_{\mu}(q,0)}= 
	\sum_{\mu}  \frac{\widetilde{K}_{\lambda' \mu'}(q)\mac{\mu'}(X;q)}{(-q)^n   q^{ 2 n(\mu')}
	 \prod_{s\in \mu,\l_{\mu}(s) = 0}(1-q^{-a_{\mu}(s)-1})}.
	\end{equation}
	\end{cor}

	\begin{proof}
	It follows from (\ref{E:specialization2}) and (\ref{E:kostkaspecialization}) that 
	$\widetilde{K}_{\lambda' \mu}(q,0)=	\widetilde{K}_{\lambda' \mu'}(0,q)=\widetilde{K}_{\lambda' \mu'}(q)$.	
	Since, 
	$\widetilde{h}_{\mu}(q,0) \widetilde{h'}_{\mu}(q,0) = (-q)^n   q^{ 2 n(\mu')} 
	\prod_{s\in \mu,\l_{\mu}(s) = 0}(1-q^{-a_{\mu}(s)-1})$, and  
	$\mac{\mu}(X;q,0)= \mac{\mu'}(X;q)$, the proof follows from Lemma \ref{L:GaHaTes}.
	\end{proof}

	\begin{thm} 
	Let $1\leq k \leq r \leq n$, and let $\mu \in Par(n,r)$. Then,
	\begin{equation*}
	T_{k,r} = \widetilde{K}_{(n-k+1,1^{k-1}) \mu'}(q).
	\end{equation*}
	\end{thm}

	\begin{proof}
	Recall that 
	\begin{equation*}
	s_{k+1,1^{n-k-1}}[\frac{X}{1-q}] = \sum_{r=k+1}^n T_{k+1,r} \frac{E_{n,r}(X)}{(q;q)_r},
	\end{equation*}
	where 
	\begin{equation*}
	T_{k+1,r}=(-1)^{k} \sum_{i=0}^k (-1)^i q^{i\choose 2} \dqbin{r}{i}.
	\end{equation*}

	Therefore, by Corollary \ref{C:schurHall} and Proposition \ref{P:EHall} we have 
	\begin{eqnarray*}
	\sum_{\mu} \frac{\widetilde{K}_{(n-k+1,1^{k-1}) \mu'}(q) \mac{\mu'}(X;q)}{\widetilde{h}_{\mu}(q,0) \widetilde{h'}_{\mu}(q,0)} 	&=& s_{k,1^{n-k}}[\frac{X}{1-q}]\\
	&=&  \sum_{r=k}^n T_{k,r}  \sum_{\mu \in Par(n,r)}
	\frac{ \mac{\mu'}(X;q)}{\widetilde{h}_{\mu}(q,0) \widetilde{h'}_{\mu}(q,0)} \\
	&=&  \sum_{\mu \in \bigcup_{r=k}^n Par(n,r)} 
	\frac{ T_{k,r}  \mac{\mu'}(X;q)}{\widetilde{h}_{\mu}(q,0) \widetilde{h'}_{\mu}(q,0)}.
	\end{eqnarray*}
	The theorem follows from comparison of the coefficients of $\mac{\mu'}(X;q)$.
	\end{proof}

\bibliographystyle{alpha}

\begin{thebibliography}{10}


\bibitem{CanLoehr06}
M.B. Can, N. Loehr, \emph{A proof of the q,t-Square conjecture.}  
J. Combin. Theory Ser. A.  113  (2006),  no. 7, 1419--1434. 

\bibitem{GaHag02}
A.M. Garsia, J. Haglund, \emph{A proof of the q,t-{C}atalan positivity conjecture.} Discrete Mathematics, \textbf{256} (2002), 677--717.

 
  
%\bibitem{GaHa96}
%A.M. Garsia, M. Haiman, 
%\emph{A remarkable $q,t$-{C}atalan sequence and $q$-Lagrange inversion}, J. Algebraic Combin. 
%\textbf{5} (1996), no.~3, 191--244.


%\bibitem{GaHa98}
%A.M. Garsia, M.Haiman, \emph{A random $q,t$-hook walk and a sum of Pieri coefficients}, 
%J. of Combinatorial Theory (A) 82 (1998) no.~1, 74-111.


%\bibitem{Gordon03}
%I. Gordon, \emph{On the quotient ring by diagonal invariants}, 
%Invent. Math.  153  (2003),  no. 3, 503--518.


%\bibitem{Gordon05}
%I. Gordon, J. Stafford, \emph{Rational Cherednik algebras and Hilbert schemes}, Adv. Math.  198  
%(2005),  no. 1, 222--274.


%\bibitem{Gott91}
%L. G{\"o}ttsche, \emph{Hilbert schemes of zero dimensional subschemes of smooth varieties}, Lecture %Notes in Math., 1572, Springer (1994).  



%\bibitem{Grot61}
%A. Grothendieck, \emph{Techniques de construction et th\'{e}or\`{e}mes
%  d'existence en g\'{e}om\'{e}trie alg\'{e}brique, {I}{V}: Les schemas de
%  {H}ilbert}, S\'{e}min.\ Bourbaki 221, IHP, Paris, 1961.
  

%\bibitem{GustavsenEtal}
%T.S. Gustavsen, D. Laksov, R.M. Skjelnes, \emph{An elementary, explicit, proof of the existence of 
%Hilbert schemes of points},  
%J. Pure Appl. Algebra  210  (2007),  no. 3, 705--720. 

\bibitem{GaHaTes}
A.M. Garsia, M. Haiman, G. Tesler \emph{Explicit plethysic formulas for Macdonald q,t-Kostka coefficients.}\\
The Andrews Festschrift. Seminaire Lotharingien 42 (1999).



%\bibitem{Hag03}
%J. Haglund, 
%\emph{Conjectured statistics for the $q,t$-Catalan numbers.}
%Adv. Math. 175 (2003), no. 2, 319--334.

%\bibitem{Hag04} 
%J. Haglund, \emph{A combinatorial model for the Macdonald polynomials}, Proc. Natl. Acad. Sci. USA  %101  (2004),  no. 46, 


\bibitem{Hag07}
J. Haglund, 
\emph{The $q,t$-Catalan numbers and the space of diagonal harmonoics.}
AMS University Lecture Series (in press).

  
\bibitem{hhlru05}
J.~Haglund, M.~Haiman, N.~Loehr, J.~B. Remmel, and A.~Ulyanov, 
\emph{A combinatorial formula for the character of the diagonal coinvariants.}
Duke Math. J. \textbf{126} (2005) no.~2, 195--232.   
  

%\bibitem{HHL05} 
%J. Haglund, M. Haiman, N. Loehr, \emph{A combinatorial formula for Macdonald polynomials},
%J. Amer. Math. Soc.  18  (2005),  no. 3, 735--761. 

%\bibitem{Hai94}
%M. ~Haiman, \emph{Conjectures on the quotient ring by diagonal invariants.} 
%J. Algebraic Combin. 3 (1994), no. 1, 17--76.


%\bibitem{Hai98}
%M. ~Haiman, \emph{$t,q$-{C}atalan numbers and the {H}ilbert scheme}, Discrete
 % Math. \textbf{193} (1998), no.~1-3, 201--224.



\bibitem{Hai03}
M. ~Haiman, \emph{Combinatorics, symmetric functions and Hilbert schemes.} 
Current Developments in Mathematics 2002, no. 1 (2002), 39--111. 

%\bibitem{Hai02}
%M. ~Haiman, \emph{Vanishing theorems and character formulas 
%for the Hilbert scheme of points in the plane}, Invent. Math. \textbf{149} (2001), 371--407.


%\bibitem{Hai01}
%M. Haiman, \emph{Notes on Macdonald polynomials and the geometry of Hilbert schemes},
%Symmetric Functions 2001:  Proceedings of the NATO Advanced Study Institute held in Cambridge, 
%Kluwer, Dordrecht (2002) 1--64.


\bibitem{LoehrWarrington}
N. Loehr, G. Warrington, \emph{Square $q,t$-lattice paths and $\nabla(p\sb n)$.} Trans. Amer. Math. Soc. 359 no. 2,  (2007), 649--669. 





\bibitem{Mac95}
I.G. Macdonald, \emph{Symmetric functions and {H}all polynomials.} Second Ed.,
Oxford University Press, New York, (1995).


%\bibitem{Stembridge94}
%J. R. Stembridge, \emph{Some particular entries of the two-parameter Kostka Matrix.} Proc. of the 
%American Math. Soc. Vol. 121, No.2.2, pp. 367-373.


\end{thebibliography}
 
\end{document}